\documentclass[3p]{classe}
\usepackage{amssymb, amsmath, amscd, amsthm}
\usepackage{lineno}
\usepackage{amsfonts}
\usepackage[all]{xy}
\usepackage{graphicx}

\numberwithin{equation}{section}

\newtheorem{thm}{Theorem}[section]

\newtheorem{prop}[thm]{Proposition}
\theoremstyle{definition}
\newtheorem{defn}[thm]{Definition}
\theoremstyle{remark}
\newtheorem{rem}[thm]{Remark}
\numberwithin{equation}{subsection}
\newtheorem{ex}[thm]{Example}

\newcommand{\R}{\mathbb{R}}

\newcommand{\K}{\mathbb{K}}
\newcommand{\Z}{\mathbb{Z}}
\newcommand{\E}{\mathbb{E}}

\newcommand{\p}{\varphi}

\newcommand{\im}{{\mathrm{im}}}

\newcommand{\Homeo}{\operatorname{Homeo}}
\newcommand{\interior}{\operatorname{int}}

\usepackage{amssymb}

\begin{document}

\begin{frontmatter}

\title{$G$-invariant persistent homology}

\author{Patrizio Frosini, Department of Mathematics and ARCES, University of Bologna}
\ead{patrizio.frosini@unibo.it}
\address{Piazza di Porta San Donato 5, 40126, Bologna, Italy}

\begin{abstract}
Classical persistent homology is a powerful mathematical tool for shape comparison. Unfortunately, it is not tailored to study the action of transformation groups that are
different from the group $\Homeo(X)$ of all
self-homeomorphisms of a topological space $X$. This fact restricts its use in applications.
In order to obtain better lower bounds for the natural pseudo-distance $d_G$ associated with a group $G\subset \Homeo(X)$,
we need to adapt persistent homology and consider $G$-invariant
persistent homology.
Roughly speaking, the main idea consists in defining persistent homology by means of a set of chains that is invariant
under the action of $G$. In this paper we formalize this idea, and prove the stability of the persistent Betti number functions in $G$-invariant
persistent homology with respect to the natural pseudo-distance $d_G$. We also show how $G$-invariant
persistent homology could be used in applications concerning shape comparison, when the invariance group is a proper subgroup of the group of all self-homeomorphisms of a topological space.
In this paper we will assume that the space $X$ is triangulable, in order to guarantee that the persistent Betti number functions are finite without using any tameness assumption. 
\end{abstract}

\begin{keyword}
Natural pseudo-distance \sep filtering function \sep group action \sep  lower bound \sep stability \sep  shape comparison
\MSC[2010]  Primary 55N35 \sep Secondary 68U05
\end{keyword}

\end{frontmatter}

\section{Introduction}
\label{Introduction}
In many applied problems we are interested in comparing two $\R^k$-valued functions defined on a topological space, up to a certain group of tranformations.
As an example, we can think of the case of taking pictures of two objects $A$ and $B$ from every possible oriented direction (at a constant distance) and comparing the sets of images we get. In such a case the image $I(v)$ taken from the oriented direction of a unit vector $v$ can be approximated by a point in $\R^k$. This point describes a matrix $M(v)$, which represents the grey levels on a grid discretizing the image $I(v)$.
Our global measurement is a function $\varphi:S^2\to\R^k$, taking each oriented direction $v\in S^2\subset \R^3$ to the vector $\p(v)$ describing the matrix $M(v)$, associated with the picture $I(v)$ that we get from that oriented direction. In this case the position of the examined objects cannot be predetermined but we can control the direction of the camera that takes the pictures. As a consequence, two different sets of pictures (described by two different functions $\p,\psi:S^2\to\R^k$) can be considered similar if an orientation-preserving rigid motion $g$ of $S^2$ exists, such that the picture of  $A$ taken from the oriented direction of the unit vector $v$ is similar to the picture of  $B$ taken from the oriented direction of the unit vector $g(v)$, for every $v\in S^2$. Formally speaking,  the two different sets of pictures
can be considered similar if $\inf_{g\in R({S^2})}\max_{v\in S^2}\left\|\p(v)-\psi(g(v))\right\|_\infty$  is small, where $R({S^2})$ denotes the group of orientation-preserving isometries of $S^2$ and $\|\cdot\|_\infty$ is the max-norm.

The previous example illustrates the use of the following definition, where $C^0(X,\R^k)$ represents the set of all continuous functions from $X$ to $\R^k$. These functions are called \emph{$k$-dimensional filtering functions} on the topological space $X$.

In this paper we will assume that the space $X$ is triangulable. This assumption allows to guarantee that the persistent Betti number functions (PBNFs) are finite without using any tameness assumption (cf. Theorem 2.3 in \cite{CeDFFe13}). The assumption that the PBNFs are finite is necessary to our treatment. We could weaken the assumption that $X$ is triangulable and consider a compact and locally contractible subspace of $\R^n$ (cf.  \cite{CaLa11}), but we preferred to refer to an assumption that is usual for the community interested in persistent homology. 

\begin{defn}\label{defdG}
Let $X$ be a triangulable space. Let $G$ be a subgroup of the group $\Homeo(X)$ of all homeomorphisms $f:X\to X$. The pseudo-distance $d_G:C^0(X,\R^k)\times C^0(X,\R^k)\to\R$ defined by setting $$d_G(\p,\psi)=\inf_{g \in G}\max_{x \in X}\left\|\p(x)-\psi(g(x))\right\|_\infty$$ is called the \emph{natural pseudo-distance associated with the group $G$}.
\end{defn}

The previous definition generalizes the concept of natural pseudo-distance studied in \cite{FrMu99,DoFr04,DoFr07,DoFr09,Fa11} to the case $G\neq \Homeo(X)$, and is a particular case of the general setting described in \cite{FrLa11}. The case that $G$ is a proper subgroup of $\Homeo(X)$ is also examined in \cite{Ca10,CaDiLa12}, and in \cite{Fr90} for the case of the group of diffeomorphisms (in an infinite dimensional setting).

The pseudo-distance $d_G$ is difficult to compute. Fortunately, if $G=\Homeo(X)$, persistent homology can be used to obtain lower bounds for $d_G$. For example, if we denote by $D_{match}$ the matching distance between the $n$-th persistent Betti number functions $\rho_n^\p$ and $\rho_n^\psi$ of the functions $\p$ and $\psi$, we have that $D_{match}(\rho_n^\p,\rho_n^\psi)\le d_{\Homeo(X)}(\p,\psi)$ (cf. \cite{BiCeFrGiLa08,CeDFFe13}).

\begin{rem}\label{rembd}
In literature concerning persistent homology, the expression \emph{matching distance} (a.k.a. bottleneck distance) usually denotes a metric between persistence diagrams. However, each persistence diagram represents just one persistent Betti number function, provided that two persistent Betti number functions are considered equivalent if they differ in a subset of their domain that has a vanishing measure. As a consequence, the matching distance can be seen as a metric between persistent Betti number functions. In this paper we shall use the expression \emph{matching distance} in this sense.
\end{rem}

For more details about persistent homology and its applications we refer the reader to \cite{CaZo09,CaZo*05,ChCo*09,EdHa08,Gh08}.

A natural question arises: How could we obtain a lower bound for $d_G$ in the general case $G\neq \Homeo(X)$? Does an analogue of the concept of persistent Betti number function exist, suitable for getting a lower bound for $d_G$?
Since $d_{\Homeo(X)}(\p,\psi)\le d_G(\p,\psi)$, one could think of using the classical lower bounds for the natural pseudo-distance $d_{\Homeo(X)}$ in order to get lower bounds for the pseudo-distance $d_G$.
Before proceeding we illustrate two examples, showing that in some cases this choice is not useful.

\begin{ex}\label{ex1}
Let us consider an experimental setting where a robot is in the middle of a room, measuring its distance from the surrounding walls by a sensor, for each oriented direction. This measurement can be formalized by a function
$\xi:S^1 \to \R$, where $\xi(v)$ equals minus the distance from the wall in the oriented direction represented by the unit vector $v$, for each $v\in S^1$. Figure~\ref{rooms} represents two instances $\p$ and $\psi$ of the function $\xi$ for two different shapes of the room. Let $R(S^1)$ denote the group of orientation-preserving rigid motions of $S^1\subset \R^2$.
We observe that a homeomorphism $f:S^1\to S^1$ exists, such that $\p=\psi\circ f$ and $f\notin R(S^1)$.
It follows that $d_{\Homeo(S^1)}(\p,\psi)=0$,
so that the direct application of classical persistent homology does not give a positive lower bound for $d_{R(S^1)}(\p,\psi)$,
while we will see that $d_{R(S^1)}(\p,\psi)>0$.
\end{ex}

\begin{figure}[htbp]
\begin{center}
\includegraphics[width=16cm]{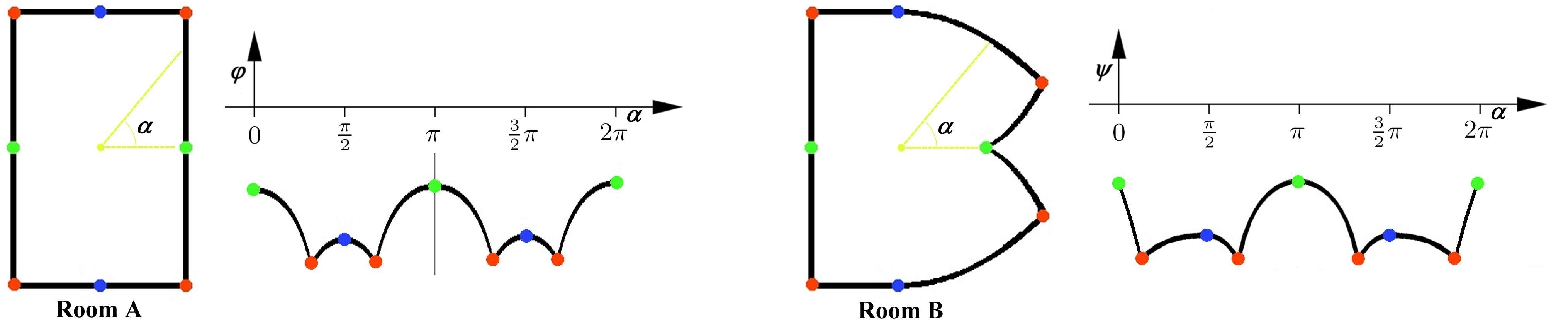}
\caption{Two rooms and the respective functions $\p,\psi$, representing minus the distance between the center and the walls. $S^1$ is identified with the interval $[0,2\pi]$.}
\label{rooms}
\end{center}
\end{figure}

\begin{ex}\label{ex2}
Let us consider the functions $\varphi_\mathtt{A},\varphi_\mathtt{D},\varphi_\mathtt{O},\varphi_\mathtt{P},\varphi_\mathtt{Q},\varphi_\mathtt{R}$ from the unit disk $D^2\subset \R^2$ to the real numbers, representing images of the letters $\mathtt{A}, \mathtt{D}, \mathtt{O}, \mathtt{P}, \mathtt{Q}, \mathtt{R}$.  For each letter $Y\in\{\mathtt{A}, \mathtt{D}, \mathtt{O}, \mathtt{P}, \mathtt{Q}, \mathtt{R}\}$, the function $\varphi_Y:D^2\to\R$ describes the grey level at each point of the topological space $D^2$, with reference to the considered instance of the letter $Y$ (see Figure~\ref{letters}). Black and white correspond to the values $0$ and $1$, respectively (so that light grey corresponds to a value close to $1$). It is easy to recognize that for each pair $(Y,Y')$ with $Y,Y'\in\{\mathtt{A}, \mathtt{D}, \mathtt{O}, \mathtt{P}, \mathtt{Q}, \mathtt{R}\}$ a homeomorphism $h:D^2\to D^2$ exists such that the max-distance between the functions $\varphi_Y,\varphi_{Y'}$ vanishes. This is due to the fact that the letters $\mathtt{A}, \mathtt{D}, \mathtt{O}, \mathtt{P}, \mathtt{Q}, \mathtt{R}$ are homeomorphic to each other.
It follows that $d_{\Homeo(D^2)}(\varphi_Y,\varphi_{Y'})$ vanishes. As a consequence, the distance between the classical persistence diagrams of $\varphi_Y$ and $\varphi_{Y'}$ vanishes, too. This proves that the direct application of classical persistent homology is not of much use in this example.
\end{ex}

\begin{figure}[htbp]
\begin{center}
\includegraphics[width=12cm]{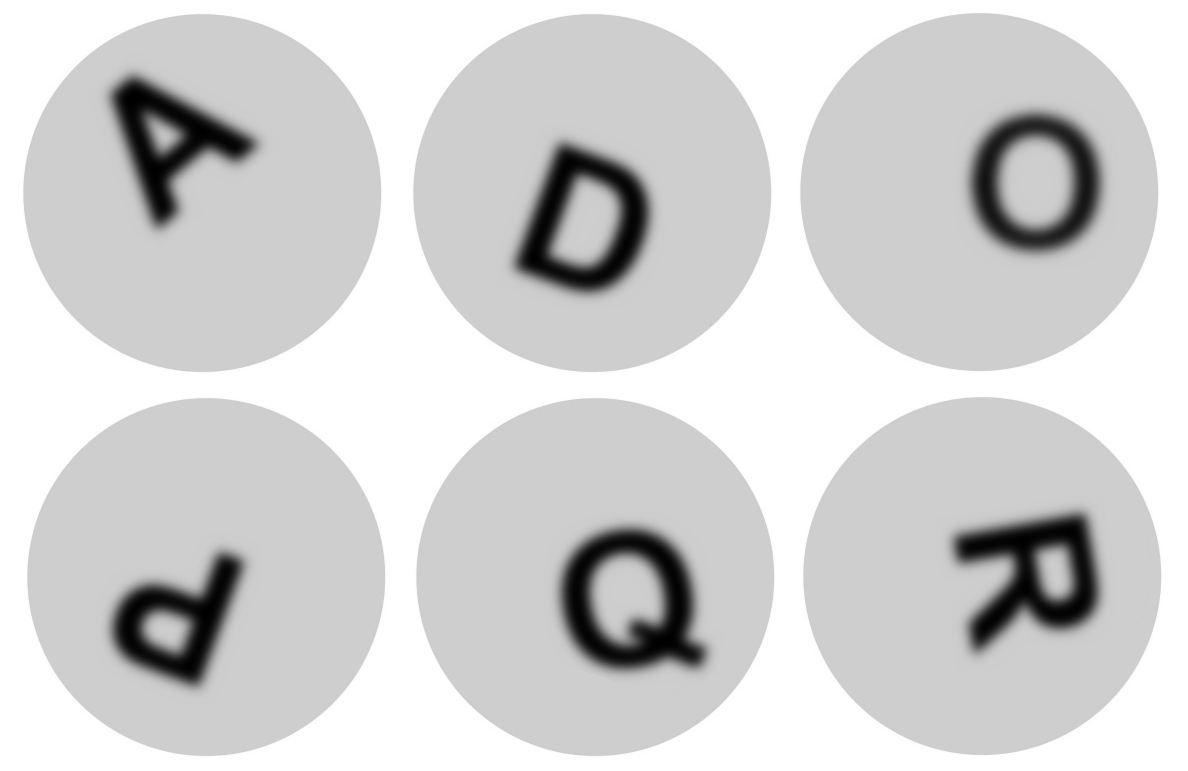}
\caption{Examples of letters $\mathtt{A}, \mathtt{D}, \mathtt{O}, \mathtt{P}, \mathtt{Q}, \mathtt{R}$ represented by functions $\varphi_\mathtt{A},\varphi_\mathtt{D},\varphi_\mathtt{O},\varphi_\mathtt{P},\varphi_\mathtt{Q},\varphi_\mathtt{R}$ from the unit disk $D^2\subset \R^2$ to the real numbers.  Each function $\varphi_Y:D^2\to\R$ describes the grey level at each point of the topological space $D^2$, with reference to the considered instance of the letter $Y$. Black and white correspond to the values $0$ and $1$, respectively (so that light grey corresponds to a value close to $1$).}
\label{letters}
\end{center}
\end{figure}

One could think of solving the problem described in the two previous examples by using other filtering functions. Unfortunately, this is not always easy to do. To make this point clear, think of acquiring data by magnetic resonance imaging (MRI). Asking for further filtering functions means asking for new measurements, of similar or different kind. This approach could be expensive or impractical. Furthermore, choosing the data we have to manage is not allowed, in many applications.

Moreover, in the fortunate case that we can choose the filtering function, another difficulty arises. It consists in the fact that shape comparison is usually based on judgements of experts, expressed by invariance properties. As an example, the expert can say that rotation and scaling are not important in the considered field of research. On one hand, we observe that it is not easy to translate the invariance properties expressed by the expert into the choice of a filtering function. On the other hand, it is quite natural to try to  directly insert the information given by the expert into our theoretical setting.
In this paper we will show that we can do that. Indeed, we can adapt persistent homology in order to obtain the invariance with respect to the action of a given group $G$ of homeomorphisms. This allows us to obtain a theory that can give a positive lower bound for $d_G$, in the previous examples (and in many similar cases, where a direct application of classical persistent homology is not of much use).

We are going to describe this idea in the next section.

\section{Adapting persistent homology to the group $G$}

This section is devoted to the introduction of some abstract definitions and the statement of a general result. In the next sections we will show how these concepts can be put into effect.

Shape comparison is commonly based on comparing properties (usually described by $\R^k$-valued functions) with respect to the action of a transformation group. Let us interpret these concepts in a homological setting. Before proceeding, let us fix a chain complex $(C,\partial)$ over a field $\K$ (so that each group of $n$-chains $C_n$ is a vector space).
We consider the partial order $\preceq$ on $\R^k$ defined by setting
$(u_1,\ldots,u_k)\preceq (v_1,\ldots,v_k)$ if and only if  $u_j\le v_j$ for every $j\in \{1,\ldots,k\}$.

\begin{defn} \label{defFF}
Let $(C,\partial)$ be a chain complex over a field $\K$.
Assume a function $\bar\varphi=(\bar\varphi_1,\ldots,\bar\varphi_k):\bigcup_n C_n\to \R^k\cup (-\infty,\ldots,-\infty)$ is given, such that
\begin{description}
\item[$i)$] $\bar\varphi$ takes the null chain $\mathbf{0}\in C_n$ to the $k$-tuple $(-\infty,\ldots,-\infty)$, for every $n\in \Z$;
\item[$ii)$] $\bar\varphi(\partial c)\preceq \bar\varphi(c)$ for every $c\in \bigcup_n C_n$;
\item[$iii)$] $\bar\varphi(\lambda c)=\bar\varphi(c)$ for every $c\in \bigcup_n C_n$, $\lambda\in\K$, $\lambda\neq 0$;
\item[$iv)$] $\bar\varphi_j(c_1+c_2)\le\max\left(\bar\varphi_j(c_1),\bar\varphi_j(c_2)\right)$ for every
$c_1,c_2\in C_n$ with  $n\in \Z$, and every $j\in \{1,\ldots,k\}$.
\end{description}
We shall say that $\bar\varphi$ is a \emph{filtering function on the chain complex $(C,\partial)$}.
\end{defn}

\begin{defn}\label{defGICC}
Let $(C,\partial)$ be a chain complex over a field $\K$.
Let us assume that a group $G$ is given,
such that $G$ acts linearly on each vector space $C_n$ and its action commutes with $\partial$, i.e., $\partial \circ g=g\circ \partial$ for every $g\in G$ (in particular, every $g\in G$ is a chain isomorphism from $C$ to $C$).
The chain complex $(C,\partial)$ will be said a \emph{$G$-chain complex}. We shall call the group $H_n(C):=\ker \partial_n/ {\rm im}\ \partial_{n+1}$ the \emph{$n$-th homology group associated with the $G$-chain complex $(C,\partial)$}.
\end{defn}

We observe that the group $G$ acts on the kernel and image whose quotient is the group $H_n(C)$. As a consequence, $G$ also acts on the homology group.

In the previous definition we do not specify how the action of $G$ on each vector space $C_n$ is chosen, confining ourselves to assume that this action is linear and commutes with $\partial$. In the next section,
$C$ will be the singular chain complex of a triangulable space $X$ over a field $\K$, and
$G$ will be assumed to be a subgroup of $\Homeo(X)$. In that section, the action of each $g\in G$ on each singular simplex in $X$ will be given by the usual composition of functions. For more details about $G$-complexes and equivariant homology we refer the interested reader to \cite{Br72,Il73,tD87,Wi75}.

Now, let us assume that $(C,\partial)$ is a $G$-chain complex, endowed with a filtering function $\bar\varphi$. For every $u\in \R^k$ we can consider the chain subcomplex $C^{\bar\varphi\preceq u}$ of $C$ defined by setting $C^{\bar\varphi\preceq u}_n:=\{c\in C_n:\bar\varphi(c)\preceq u\}$ and restricting $\partial$ to $C^{\bar\varphi\preceq u}$. $C^{\bar\varphi\preceq u}$ is a subcomplex of $C$ because of the properties in Definition~\ref{defFF} (in particular, $\partial(C^{\bar\varphi\preceq u}_{n+1})\subseteq C^{\bar\varphi\preceq u}_n$). We observe that $C^{\bar\varphi\preceq u}$ will not be a $G$-chain complex, since $g(C^{\bar\varphi\preceq u}_n)\not\subseteq C^{\bar\varphi\preceq u}_n$, in general. For the sake of simplicity, we will use the symbol $\partial$ in place of $\partial_{|C^{\bar\varphi\preceq u}}$.

\begin{defn}\label{defCSCPHI}
The chain complex $\left(C^{\bar\varphi\preceq u},\partial\right)$
will be called the \emph{chain subcomplex of $(C,\partial)$ associated with the value $u\in \R^k$, with respect to the filtering function $\bar\varphi$}.
\end{defn}

We refer to  \cite{KaMiMr04} for the definition of chain subcomplex.

Now we can define the concept of the $n$-th persistent homology group of $(C,\partial)$, with respect to $\bar\varphi$.

\begin{defn}\label{defPHG}
If $u=(u_1,\ldots,u_k), v=(v_1,\ldots,v_k) \in \R^k$ and $u\prec v$ (i.e., $u_j< v_j$ for every index $j$), we can consider the inclusion $i$ of the chain complex
$C^{\bar\p\preceq u}$ into the chain complex $C^{\bar\p\preceq v}$. Such an inclusion
 induces a homomorphism
 $i_*: H_n\left(C^{\bar\p\preceq u}\right) \to H_n\left(C^{\bar\p\preceq v}\right)$.
 We shall call the group $PH^{\bar\p}_n(u,v):=i_*\left(H_n\left(C^{\bar\p\preceq u}\right)\right)$ the \emph{$n$-th persistent homology
 group of  the $G$-chain complex $C$, computed at the point $(u,v)$ with respect to the filtering function $\bar\varphi$}.
 The rank $\rho_n^{\bar\p}(u,v)$ of this group will be called \emph{the $n$-th persistent Betti
 number function (PBNF) of  the $G$-chain complex $C$, computed at the point $(u,v)$ with respect to the filtering function $\bar\varphi$}.
\end{defn}

The key property of $PH^{\bar\p}_n$ is the invariance expressed by the following result.

\begin{thm}\label{propinv}
If $g$
is a chain isomorphism from $C$ to $C$
and $u,v\in \R^k$ with $u\prec v$, the groups $PH^{\bar\p\circ g}_n(u,v)$ and $PH^{\bar\p}_n(u,v)$ are isomorphic.
\end{thm}

\begin{proof}
We define a map $F:PH^{\bar\p\circ g}_n(u,v)\to PH^{\bar\p}_n(u,v)$ in the following way. Let us consider an element $z\in PH^{\bar\p\circ g}_n(u,v):=i_*\left(H_n\left(C^{\bar\p\circ g\preceq u}\right)\right)$. By definition, a cycle $c\in C_n^{\bar\p\circ g\preceq u}$ exists, such that $z$ is the equivalence class $[c]_v$ of $c$ in $H_n\left(C^{\bar\p\circ g\preceq v}\right)$.
We observe that $g(c)\in C_n^{\bar\p\preceq u}$ and the equivalence class $[g(c)]_v$ of $g(c)$ in $H_n\left(C^{\bar\p\preceq v}\right)$ belongs to $PH^{\bar\p}_n(u,v):=i_*\left(H_n\left(C^{\bar\p\preceq u}\right)\right)$.
We set $F(z)=[g(c)]_v$.

If $c'\in C_n^{\bar\p\circ g\preceq u}$ is another cycle such that $z=[c']_v\in H_n\left(C^{\bar\p\circ g\preceq v}\right)$, then a chain $\gamma \in C_{n+1}^{\bar\p\circ g\preceq v}$ exists, such that $c'-c=\partial\gamma$.
We observe that $g(\gamma)\in C_{n+1}^{\bar\p\preceq v}$. The inequality $\bar\varphi(\partial (g(\gamma)))\preceq \bar\varphi(g(\gamma))$ (see Definition~\ref{defFF}) implies that  $\partial (g(\gamma))\in C_{n}^{\bar\p\preceq v}$. As a consequence, $[g(c')]_v=[g(c+\partial\gamma)]_v=[g(c)+g(\partial\gamma)]_v=[g(c)+\partial(g(\gamma))]_v=[g(c)]_v+[\partial(g(\gamma))]_v=[g(c)]_v$.
These equalities follow from the fact that $g$ is a chain isomorphism.
This proves that $F$ is well defined.

Let $z_1=[c_1]_v, z_2=[c_2]_v\in PH^{\bar\p\circ g}_n(u,v)$, with $c_1,c_2\in C_n^{\bar\p\circ g\preceq u}$. We observe that $g(c_1),g(c_2)\in C_n^{\bar\p\preceq u}$. From the linearity of $g$, it follows that $g(\lambda_1 c_1+\lambda_2 c_2)=\lambda_1 g(c_1)+\lambda_2 g(c_2)\in C_n^{\bar\p\preceq u}$, for every $\lambda_1,\lambda_2\in \K$.
Hence, we have that $F(\lambda_1 z_1+\lambda_2 z_2)=F(\lambda_1 [c_1]_v+\lambda_2 [c_2]_v)=F([\lambda_1 c_1+\lambda_2 c_2]_v)=[g(\lambda_1 c_1+\lambda_2 c_2)]_v=\lambda_1 [g(c_1)]_v+\lambda_2 [g(c_2)]_v=\lambda_1 F([c_1]_v)+\lambda_2 F([c_2]_v)=\lambda_1 F(z_1)+\lambda_2 F(z_2)$. Therefore, $F$ is linear.

Furthermore, if $F(z_1)=F(z_2)$ then $[g(c_1)]_v=[g(c_2)]_v$, so that a chain $\hat\gamma \in C_{n+1}^{\bar\p\preceq v}$ exists, such that $g(c_1-c_2)=g(c_1)-g(c_2)=\partial\hat\gamma$. Moreover, $g^{-1}(\hat\gamma)\in C_{n+1}^{\bar\p\circ g\preceq v}$.
It follows that $c_1-c_2=g^{-1}(\partial\hat\gamma)=\partial\left(g^{-1}(\hat\gamma)\right)\in C_{n}^{\bar\p\circ g\preceq v}$, because of Definitions~\ref{defFF} and
the fact that $g$ is a chain isomorphism.
As a consequence, $[c_1]_v=[c_2]_v$. This proves that $F$ is injective.

Finally, $F$ is surjective. In order to prove this, we observe that if $w\in PH^{\bar\p}_n(u,v):=i_*\left(H_n\left(C^{\bar\p\preceq u}\right)\right)$ with the homomorphism $i_*:H_n\left(C^{\bar\p\preceq u}\right)\to H_n\left(C^{\bar\p\preceq v}\right)$  induced by the inclusion $i:C^{\bar\p\preceq u}\hookrightarrow C^{\bar\p\preceq v}$, then a chain $\hat c\in C_n^{\bar\p\preceq u}$ exists such that
$w=[\hat c]_v\in H_n\left(C^{\bar\p\preceq v}\right)$.
We have that $g^{-1}(\hat c)\in C_n^{\bar\p\circ g\preceq u}$ and $F\left([g^{-1}(\hat c)]_v\right)=[\hat c]_v=w$.

Therefore $F:PH^{\bar\p\circ g}_n(u,v)\to PH^{\bar\p}_n(u,v)$ is an isomorphism.
\end{proof}

The previous theorem justifies the name \emph{$G$-invariant persistent homology}, showing that
the PBNFs of a $G$-chain complex do not change if we replace the filtering function $\bar\p$ with the function $\bar\p\circ g$, for $g\in G$.

\section{Stability of the PBNFs with respect to $d_G$}
\label{stability}

In the previous section we have introduced some abstract definitions and a theorem. In this section we will show how we can obtain structures conforming to the previously described properties.

Let $X$ and $(S(X),\partial)$ be a triangulable space and its singular chain complex over a field $\K$, respectively.

Assume that a subgroup $G$ of the group $\Homeo(X)$ of all homeomorphisms $f:X\to X$ and a continuous function $\p=(\p_1,\ldots,\p_k):X\to \R^k$ are chosen. For every $u\in\R^k$, let us set $X^{\varphi\preceq u}:=\{x\in X:\varphi(x)\preceq u\}$.
Let us consider the action of $G$ on $S(X)$ defined by setting $g(\sigma):=g\circ \sigma$ for every $g\in G$ and every singular simplex $\sigma$ in $X$, and extending this action linearly on $S(X)$. We recall that, by definition,  every singular $n$-simplex in $X$ is a continuous function from the standard $n$-simplex
$\Delta_n$  into $X$.

Now, assume that a $G$-chain subcomplex $(\bar C,\partial)$ of the singular chain complex $(S(X),\partial)$ is given (we will show in the next section how this subcomplex can be constructed). We observe that, for every topological subspace $\bar X$ of $X$,  $(\bar C\cap S(\bar X),\partial)$ is a chain complex over the field $\K$. The symbol $\bar C\cap S(\bar X)$ denotes the chain complex $C'$ where $C'_n$ is the vector space of the singular $n$-chains in $\bar X$ that belong to $\bar C_n$.

In order to avoid ``wild'' chain complexes, we also make this assumption (see Remark~\ref{remstar} below):
\begin{description}
\item[$(*)$] If $X'$ and $X''$ are two closed subsets of $X$ with $X'\subseteq \interior(X'')$, then a topological subspace $\hat X$ of $X$ exists such that $X'\subseteq \hat X\subseteq X''$ and the homology group $H_n(\bar C\cap S(\hat X))$ is finitely generated
for every non-negative integer $n$.
\end{description}

Let us consider the set $\{\sigma_j^n\}_{j\in J}$ of all (distinct) singular $n$-simplexes in $X$. Obviously, if $X$ is not a finite topological space, $J$ will be an infinite (usually uncountable) set. Then we can endow the chain complex $\bar C$ with a filtering function $\bar\p$ in the following way. If $c$ equals the null chain in $\bar C_n$, we set $\bar\varphi(c):=(-\infty,\ldots,-\infty)$. If $c$ is a non-null singular $n$-chain, we can write $c=\sum_{r=1}^{m} a^r\sigma_{j_r}^n\in \bar C_n$
with $a^r\in \K$, $a^r\neq 0$ for every index $r$, and $j_{r'}\neq j_{r''}$ for $r'\neq r''$. This representation is said to be \emph{reduced}. 
In this case we set $\bar\varphi(c)=(u_1,\ldots,u_k)\in\R^k$, with each $u_i$ equal to the maximum of $\p_i$ on the union of the images of the singular simplexes $\sigma_{j_1}^n,\ldots,\sigma_{j_m}^n$.
In other words, $\bar\varphi(c)$ is the smallest vector $u$ such that the corresponding sublevel set $X^{\varphi\preceq u}$ contains the image of each singular simplex $\sigma_{j_r}^n$ involved in the reduced representation of $c$ that we have considered. We observe that this representation is unique up to permutations of its summands, so that $\bar\p$ is well defined.
Furthermore, the properties in Definition~\ref{defFF} are fulfilled. We shall say that the function $\bar \p$ is \emph{induced by $\p$}.

 An elementary introduction to singular homology can be found in \cite{Ha02}.

The next result has a key role in the rest of this paper and is analogous to the finiteness results proven in \cite{CeDFFe13} and \cite{CaLa11} for classical persistent homology.

\begin{prop}\label{tF}
For every $n\in \Z$ the $n$-th persistent Betti number function $\rho_n^{\bar \p}(u,v)$ of the $G$-chain complex $(\bar C,\partial)$, endowed with the filtering function $\bar\p$, is finite at each point $(u,v)$ in its domain.
\end{prop}

\begin{proof}
Since $u\prec v$ and $\p$ is continuous, we have that the set $X^{\p\preceq u}$ is closed and contained in the interior of the closed set $X^{\p\preceq v}$.
Property $(*)$ implies that a topological subspace $\hat X$ of $X$ exists such that $X^{\p\preceq u}\subseteq \hat X\subseteq X^{\p\preceq v}$ and $H_n(\bar C\cap S(\hat X))$ is finitely generated.
The inclusions $\bar C\cap S(X^{\p\preceq u})\stackrel{i}{\hookrightarrow} \bar C\cap S(\hat X)\stackrel{j}{\hookrightarrow} \bar C\cap S(X^{\p\preceq v})$ induce the homomorphisms
$H_n(\bar C\cap S(X^{\p\preceq u}))\stackrel{i_*}{\to}
H_n(\bar C\cap S(\hat X))\stackrel{j_*}{\to}
H_n(\bar C\cap S(X^{\p\preceq v}))$. Since
$\dim\im\,(j_*\circ i_*)\leq\dim\im\,j_*\le \dim H_n(\bar C\cap S(\hat X))<\infty$, we obtain that also $PH^{\bar\p}_n(u,v):=j_*\circ i_*\left(H_n\left(\bar C\cap S(X^{\p\preceq u})\right)\right)$ is finitely generated.
\end{proof}

\begin{rem}\label{remstar}
We stress the importance of the assumption $(*)$.
It allows us to avoid chain complexes like the one
where the $0$-chains are all the usual singular $0$-chains
of $X$
and the only $1$-chain is the singular
zero $1$-chain of $X$.
Obviously, this is a $G$-chain complex for any subgroup $G$ of $\Homeo(X)$.
In this case, for any pair $(P_1,P_2)$ of distinct points of the topological space $X$, there is no singular $1$-chain whose boundary is the singular $0$-chain $P_2-P_1$ (here, for the sake of simplicity, we are not distinguishing the singular $0$-simplexes from their images in $X$).
Since the boundary homomorphism from $1$-chains to $0$-chains is zero, no non-zero $0$-chain is a boundary. Hence the homology group $H_0(\bar C)$ is not finitely generated, in general,  and the property~$(*)$ does not hold.
For example, it does not hold for $X'=X''=X$,
independently of the regularity of the space $X$ (unless $X$ is a finite set).
As a consequence, the proof that we gave for
Proposition~\ref{tF} does not work, and it is easy to check that its statement
is false for the chain complex we have just described. This is the reason for which the finiteness results proven in \cite{CeDFFe13} and \cite{CaLa11} for classical persistent homology cannot be directly applied to $G$-invariant persistent homology, without assuming property~$(*)$. Finally, we observe that  $(*)$ is not as much an assumption about the regularity of the topological space $X$, but rather an assumption about the regularity of the $G$-chain complex.
\end{rem}

From now on, in order to avoid technicalities that are not relevant in this paper,
we shall consider two PBNFs equivalent if they differ in a subset of their domain that has a vanishing measure.

A standard way of comparing two classical persistent Betti number functions is the matching distance $D_{match}$, a.k.a. bottleneck distance (cf. \cite{EdHa08,CeDFFe13}). It is important to observe that, in order to define it, we need the finiteness of the persistent Betti number functions (cf. \cite{CoEdHa07}).
This distance can be applied without any modification to the case of the persistent Betti number functions of the $G$-chain complex $\bar C$, because of the finiteness stated in Proposition~\ref{tF}.

The following theorem shows that the matching distance between persistent Betti number functions of the $G$-chain complex $\bar C$ is a lower bound for the natural pseudo-distance $d_G$. In other words, a small change of the filtering function with respect to $d_G$ produces just a small change of the corresponding persistent Betti number function with respect to $D_{match}$. This property allows the use of PBNFs in real applications, where the presence of noise is unavoidable.

\begin{thm}\label{tG}
For every $n\in \Z$, let us consider the $n$-th persistent Betti number functions $\rho_n^{\bar \p}$, $\rho_n^{\bar \psi}$ of the $G$-chain complex $(\bar C,\partial)$, endowed with the filtering functions $\bar\p$ and $\bar\psi$ induced by $\p:X\to \R^k$ and $\psi:X\to \R^k$, respectively. Then $D_{match}(\rho_n^{\bar \p},\rho_n^{\bar \psi})\le d_G(\p,\psi).$
\end{thm}

\begin{proof}
We can proceed by mimicking step by step the proof of stability for ordinary persistent Betti number functions (cf. \cite{CeDFFe13}).
This is possible because that proof depends only on properties of PBNFs that are shared by both classical persistent Betti number functions and persistent Betti number functions of a $G$-chain complex endowed with a filtering function, once we have proven that the PBNFs are finite
(Proposition~\ref{tF}).
It is sufficient to replace the group $\Homeo(X)$ with the group $G\subseteq \Homeo(X)$, and the homology groups of each sublevel set $X^{\varphi\preceq u}$ with the homology groups of the $G$-chain complex $\bar C\cap S(X^{\varphi\preceq u})$. Since the only difference in the proof consists in the need to show that $G$-invariant persistent Betti number functions are finite in order to be allowed to use the matching distance $D_{match}$, we refer the reader interested in the technical details to \cite{CeDFFe13}.
\end{proof}

\section{Applications}\label{Applications}

\subsection{A first application of our method}\label{application}
In this subsection we illustrate how $G$-invariant persistent homology can be used to discriminate between the rooms described in Example~\ref{ex1}, showing that no rotation of $S^1$ changes the function $\varphi$ into $\psi$.

In order to manage this problem we can consider the chain complex $\bar C$ whose $n$-chains are all the singular $n$-chains $c\in S_n(S^1)$ for which the following property holds:
\begin{description}
\item[$(P)$] If a singular simplex $\sigma_i^n$ appears in a reduced representation of $c$ with respect to the basis $\{\sigma_j^n\}$ of $S_n(S^1)$, then the antipodal simplex $s\circ \sigma_i^n$ appears in that representation with the same multiplicity of $\sigma_i^n$, where $s$ is the antipodal map $s:S^1 \to S^1$.
\end{description}
In other words, in $\bar C$ we accept by definition only the singular chains in $S^1$ that can be written in the form $\sum_{r=1}^{m} a^r\left(\sigma_{j_r}^n+s\circ \sigma_{j_r}^n\right)$. 
It easy to check that $(\bar C,\partial)$ is a $R(S^1)$-chain subcomplex of the complex $(S(S^1),\partial)$.

  Every rotation $\rho\in R(S^1)$ commutes with the antipodal map $s$ and is a chain isomorphism from $\bar C$ to $\bar C$. Moreover, it is easy to verify that the properties in Definition
\ref{defGICC} are fulfilled, for $G=R(S^1)$ and $C=\bar C$. The chains in $\bar C$ will be called \emph{symmetric chains}.

We can prove that the property $(*)$ holds for the $R(S^1)$-chain complex that we have defined.
Let $X'$ and $X''$ be two closed subsets of $S^1$ with $X'\subseteq \interior(X'')$.
Let us set $\hat X$ equal to the $\varepsilon$-dilation\footnote{The $\varepsilon$-dilation of a subset $Y$ of a metric space $M$ is the set of points of $M$ that have a distance strictly less than $\varepsilon$ from $Y$. On $S^1\subset \R^2$ we consider the metric induced by the Euclidean metric in $\R^2$.} of $X'$ in $S^1$, choosing $\varepsilon>0$ so small that the $\hat X \subseteq \interior(X'')$.
We observe that the set $\hat X\cap s(\hat X)$ is open and $s\left(\hat X\cap s(\hat X)\right)=\hat X\cap s(\hat X)$.
Moreover, $\hat X\cap s(\hat X)$ is the union of a finite family $\mathcal{F}=\{\alpha_i\}$ of pairwise disjoint open arcs, having the property that if $\alpha_i\in \mathcal{F}$ then also $s(\alpha_i)\in \mathcal{F}$ (possibly, $\mathcal{F}=\{S^1\}$).
Now, let us consider the topological quotient space $Q$ obtained by taking all unordered pairs of antipodal points in $\hat X\cap s(\hat X)$.
We have that $Q$ is homeomorphic to the union of a finite family $\mathcal{F}'$ of pairwise disjoint  open arcs of $S^1$ (possibly, $\mathcal{F}'=\{S^1\}$), and hence the $n$-th homology group $H_n(Q)$ is finitely generated.
A chain isomorphism $F$ from $\bar C\cap S\left(\hat X\cap s(\hat X)\right)$ to $S(Q)$ exists, taking each $n$-chain $\sigma+s\circ\sigma$ to the chain given by the singular simplex $\tilde \sigma:\Delta_n\to Q$, defined by setting $\tilde\sigma(p):=\{\sigma(p),s\circ\sigma(p)\}$ for every $p\in \Delta_n$.
$F$ induces an isomorphism
from $H_n\left(\bar C\cap S\left(\hat X\cap s(\hat X)\right)\right)$ to $H_n(Q)$.
Therefore also $H_n\left(\bar C\cap S\left(\hat X\cap s(\hat X)\right)\right)$ is finitely generated. Property $(*)$ follows by observing that $\bar C\cap S\left(\hat X\cap s(\hat X)\right)=\bar C\cap S(\hat X)$.

Referring to Example~\ref{ex1}, let us consider the birth of the first homology class in the homology groups $H_0\left(\bar C^{\bar\varphi\le t}\right)$ and $H_0\left(\bar C^{\bar \psi\le t}\right)$, respectively, when the parameter $t$ increases. While the group $H_0\left(\bar C^{\bar \varphi\le t}\right)$ becomes non-trivial when $t$ reaches the value $t_0=\min\p=\min\psi$, the group $H_0\left(\bar C^{\bar \psi\le t}\right)$ becomes non-trivial when $t$ reaches a value $\bar t>\min\p=\min\psi$. This is due to the fact that the sublevel set $\{x\in S^1:\varphi(x)\le t_0\}$ contains two pairs of antipodal points, while the sublevel set $\{x\in S^1:\psi(x)\le t_0\}$ contains no pair of antipodal points (see Figure~\ref{ap}). In other words, the only points at infinity in the persistence diagrams associated with the $0$-th persistent homology groups of the $G$-chain subcomplex $\bar C$ of $S(S^1)$ with respect to $\bar \varphi$ and $\bar\psi$ are $(t_0,\infty)$ and $(\bar t,\infty)$, respectively.

It follows that the matching distance between the $0$-th persistent Betti number functions of the $R(S^1)$-chain complex $\bar C$ with respect to the filtering functions $\bar \p$ and $\bar \psi$ is at least $\bar t-t_0>0$. By applying Theorem~\ref{tG}, we obtain the inequality $d_{R(S^1)}(\p, \psi)\ge \bar t-t_0$.
In other words,  $G$-invariant persistent homology gives a non-trivial lower bound for $d_{R(S^1)}(\p, \psi)$, while the matching distance between the classical persistent Betti number functions with respect to the filtering functions $\p$ and $\psi$ vanishes. 

The interested reader can find the $0$-th persistent Betti number functions $\rho_0^{\bar\varphi}$ and $\rho_0^{\bar\psi}$ of the $R(S^1)$-chain complex $\bar C$ in Figure~\ref{PBNF_degree_0}. We notice that the persistent Betti number functions $\rho_1^{\bar\varphi}$ and $\rho_1^{\bar\psi}$ of the $R(S^1)$-chain complex $\bar C$ coincide. Indeed, $\varphi$ and $\psi$ take the same absolute maximum $\tilde t$. Hence both the groups $H_1\left(\bar C^{\bar \varphi\le t}\right)$ and $H_1\left(\bar C^{\bar \psi\le t}\right)$ becomes non-trivial (and equal to $\K$) when $t$ reaches the same value $\tilde t=\max\varphi=\max\psi$. After that change, no further change happens. As a consequence, the persistent Betti number functions in degree $1$ of the $R(S^1)$-chain complex $\bar C$ with respect to the filtering functions $\bar \p$ and $\bar \psi$ coincide.

\begin{figure}[htbp]
\begin{center}
\includegraphics[width=16cm]{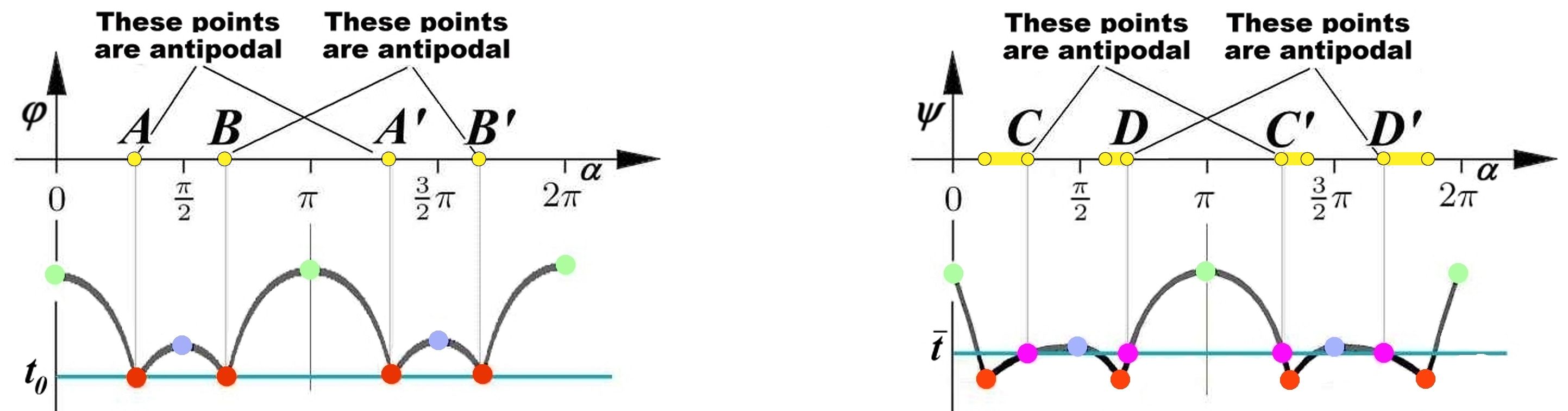}
\caption{The sublevel sets of the filtering functions  $\varphi,\psi$ cited in Example~\ref{ex1},
respectively
for the levels $t_0$ and $\bar t$.}
\label{ap}
\end{center}
\end{figure}

\begin{figure}[htbp]
\begin{center}
\includegraphics[width=16cm]{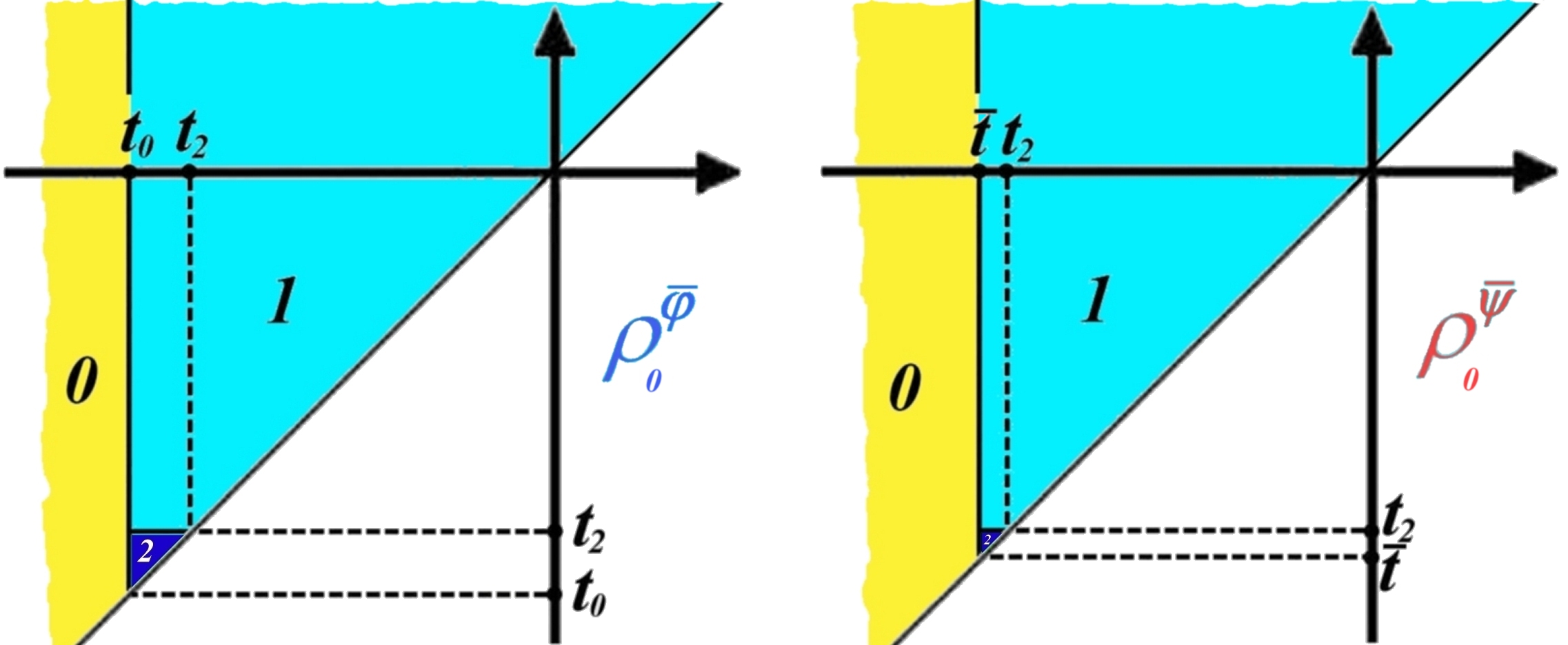}
\caption{The $0$-th persistent Betti number functions $\rho_0^{\bar\varphi}$ and $\rho_0^{\bar\psi}$ of the $R(S^1)$-chain complex $\bar C$, corresponding to the filtering functions  $\varphi,\psi$ cited in Example~\ref{ex1}. In each part of the domain, the value taken by the PBNF is displayed. Observe that in both figures a small triangle is present, at which the persistent Betti number function takes the value $2$.}
\label{PBNF_degree_0}
\end{center}
\end{figure}

\begin{rem}\label{rem2} As an alternative approach to the problem of comparing two filtering functions $\p,\psi:X\to \R$, the reader could think of using the well known concept of equivariant homology (cf.  \cite{Wi75}). In other words, in the case that $G$ acts freely on $X$, one could think of considering the topological quotient space
$X/G$, endowed with the filtering functions $\hat \p,\hat \psi$ that take each orbit $\omega$ of the group $G$ to the maximum of $\p$ and $\psi$ on $\omega$, respectively.
We observe that this approach would not be of help in the case illustrated in Example~\ref{ex1}, since the quotient of $S^1/R(S^1)$ is just a singleton. As a consequence, if we considered two filtering functions $\p,\psi:S^1\to\R$ with $\max\p=\max\psi$, the persistent homology of the induced functions $\hat\p,\hat\psi:S^1/R(S^1)\to \R$ would be the same. For more details about $G$-complexes and equivariant homology we refer the interested reader to \cite{Br72,Il73,tD87}.
\end{rem}

\subsection{A generalization of our technique}\label{generalization}
The approach that we have illustrated in the previous subsection can be generalized to triangulable spaces different from $S^1$ and invariance groups $G$ that are different from the group of rotations.  The main idea consists in looking for another subgroup $H$ of $\Homeo(X)$ such that
   \begin{enumerate}
\item  \label{a} $H$ is finite (i.e. $H=\{h_1,\ldots,h_r\}$);
\item \label{b} $g\circ h\circ g^{-1}\in H$ for every $g\in G$ and every $h\in H$.
\end{enumerate}

Due to the finiteness of $H$, the property~\ref{b} implies that the restriction to $H$ of the conjugacy action of each $g\in G$ is a permutation of $H$.

The legitimate $n$-chains in our chain complex $\bar C$ are defined to be the linear combinations of ``elementary'' singular chains $c$ that can be written as $c=\sum_{i=1}^r h_i\circ \sigma$, where $\sigma:\Delta_n\to X$ is a singular $n$-simplex in $X$. Because of the property~\ref{b} and the linearity of the action of each $g\in G$,
$g\left(\sum_{i=1}^r h_i\circ \sigma\right)=\sum_{i=1}^r g\circ h_i\circ \sigma=\sum_{i=1}^r (g\circ h_i\circ g^{-1})\circ (g\circ \sigma)=\sum_{i=1}^r h_i\circ (g\circ \sigma)$ is another legitimate chain in our chain complex  $\bar C$, so that $\bar C$ results to be a $G$-chain complex.
In Example~\ref{ex1}, we have chosen $H=\{id,s\}\subset G=R(S^1)$, where $s$ is the antipodal simmetry. We recall that the filtering function $\varphi:X\to\R^k$ induces a filtering function $\bar\varphi=(\bar\varphi_1,\ldots,\bar\varphi_k)$ on the set of legitimate chains, where $\bar\varphi(c)$ is the smallest vector $u$ such that the corresponding sublevel set $X^{\varphi\preceq u}$ contains the image of each singular simplex involved in a reduced representation of $c$, for every non-null chain $c\in\bar C_n$.

If $G$ is Abelian, a simple way of getting a subgroup $H$ of $\Homeo(X)$ verifying the properties~\ref{a} and \ref{b} consists in setting $H$ equal to a finite subgroup of $G$. This is exactly what we did in Example~\ref{ex1},  setting  $H=\{id,s\}\subset G=R(S^1)$.

If $G$ is finite, a trivial way of getting a subgroup $H$ of $\Homeo(X)$ verifying the properties~\ref{a} and \ref{b} consists in setting $H=G$. This choice leads to consider the quotient space $X/G$, provided that $G$ acts freely on $X$.

However, we stress the fact that our approach is far more general.
Indeed, in both Examples~\ref{ex1} and \ref{ex2}, if we set $G$ equal to the (Abelian and finite) group generated by the reflections with respect to the coordinate axes, we could choose $H$  equal to the group generated by the counterclockwise rotation of $2\pi/m$ radians (where $m$ denotes a fixed natural number greater than $2$). It is interesting to observe that in this case, if the homeomorphism $g$ reverses the orientation, then the conjugacy action $h\mapsto g\circ h\circ g^{-1}$ is not the identity, since it takes each homeomorphism $h$ to its inverse $h^{-1}$. Furthermore, $H\not\subseteq G$.


\begin{ex}\label{ex3}
On the basis of the remarks that we have made in this subsection, we can give another example concerning our adaptation of persistent homology to invariance groups. Let us consider $S^2=\{(x_1,x_2,x_3)\in\R^3:x_1^2+x_2^2+x_3^2=1\}$, and the two sets $X_+:=\{(x_1,x_2,x_3,x_4)\in\R^4:x_1^2+x_2^2+x_3^2=1,x_4=1\}$, $X_-:=\{(x_1,x_2,x_3,x_4)\in\R^4:x_1^2+x_2^2+x_3^2=1,x_4=-1\}$. Let us consider also the topological space $\bar X=X_+\cup X_-$, with the topology (and the metric) induced by its embedding in $\R^4$. From the topological point of view, $\bar X$ is the disjoint union of two copies of $S^2$.

Let $\bar G$ be the group of all isometries $g:\bar X\to \bar X$ that can be represented (with a little abuse of notation) as $g(x_1,x_2,x_3,x_4)=\left(\tilde g(x_1,x_2,x_3),x_4\right)$ for every $(x_1,x_2,x_3,x_4)\in \bar X$, where $\tilde g$ is an isometry of $S^2$. In plain words, these are the isometries that act similarly on $X_+$ and $X_-$. Assume that we are interested in the comparison of continuous functions from $\bar X$ to $\R$ with respect to the group $\bar G$.
In order to proceed, we have to choose a group $\bar H$ verifying the properties 1 and 2 in this subsection. For instance, we can set $\bar H$ 
equal to the group $\{id,{\bar s}\}$, generated by the map ${\bar s}:\bar X\to \bar X$ defined by setting ${\bar s}(x_1,x_2,x_3,x_4)=(x_1,x_2,x_3,-x_4)$ for every $(x_1,x_2,x_3,x_4)\in \bar X$. Following the procedure illustrated in this subsection, we obtain a $\bar G$-chain complex $\bar C$. By definition, in $\bar C$ we accept only chains that can be written in the form $\sum_{r=1}^{m} a^r\left(\sigma_{j_r}^n+{\bar s}\circ \sigma_{j_r}^n\right)$, with respect to the basis $\{\sigma_j^n\}$ of $S_n(\bar X)$.

In other words, since $\bar H$ acts freely on $\bar X$, the $n$-chains in $\bar C$ are the singular $n$-chains $c\in S_n(\bar X)$ for which the following property holds:
\begin{description}
\item[$(P')$] If a singular simplex $\sigma_i^n$ appears in a reduced representation of $c$ with respect to the basis $\{\sigma_j^n\}$ of $S_n(\bar X)$, then the simplex ${\bar s}\circ \sigma_i^n$ appears in that representation with the same multiplicity of $\sigma_i^n$.
\end{description} 

Every $g\in \bar G$ commutes with the map ${\bar s}$ and is a chain isomorphism from $\bar C$ to $\bar C$. Moreover, it is easy to verify that the properties in Definition
\ref{defGICC} are fulfilled, for $C=\bar C$.

Now, we want to prove that $\bar C$ satisfies 
the property $(*)$ described in Section~\ref{stability}.
Let us consider a sufficiently small $\epsilon >0$ such that a finite cover $\{B_1,\ldots,B_l\}$ of $X'$ exists, where each $B_i$ is a closed ball of radius $\epsilon$, contained in $X''$. We want to prove that the compact set $\hat X:=\bigcup_{i=1}^lB_i$ verifies the statement described in the property $(*)$. 

First of all, we observe that an elementary $n$-chain $c=\sigma+\bar s\circ \sigma \in \bar C$ belongs to $S(\hat X)$ if and only if $\sigma$ is a singular $n$-simplex in $\hat Y:=\hat X\cap \bar s(\hat X)=\left(\bigcup_{i=1}^lB_i\right)\bigcap \bar s\left(\bigcup_{i=1}^lB_i\right)$. Therefore, $S(\hat X)=S(\hat Y)$. We notice that $\bar s(\hat Y)=\hat Y$, so that $\bar H$ induces an action on $\hat Y$. 

Let us consider the map $F$ that takes each elementary $n$-chain $c=\sigma+\bar s\circ \sigma\in\bar C$ to the singular $n$-simplex $\tilde \sigma:\Delta_n\to \bar X/\bar H$, defined by setting $\tilde\sigma(p):=\{\sigma(p),\bar s\circ\sigma(p)\}$ for every $p\in \Delta_n$. Since $\bar H$ acts freely on $\bar X$, every $n$-chain in $\bar C$ admits a unique representation as a linear combination of distinct elementary singular chains, where each elementary chain $c$ can be written as $c=\sigma+\bar s\circ \sigma$. Therefore, $F$ extends to a unique chain map $\bar F:\bar C\to S(\bar X/\bar H)$. It is easy to check that $\bar F$ is a chain isomorphism and that 
$F(\bar C\cap S(\hat Y))=S(\hat Y/\bar H)$.
Therefore,  the homology group $H_n(\bar C\cap S(\hat X))=H_n(\bar C\cap S(\hat Y))$ is isomorphic to the homology group $H_n(S(\hat Y/\bar H))$. 

A homeomorphism $f:
\bar X/\bar H\to S^2$ exists, which takes each orbit $\{(x_1,x_2,x_3,1),(x_1,x_2,x_3,-1)\}\in \bar X/\bar H$ to the point 
$(x_1,x_2,x_3)\in S^2$. The homeomorphism $f$ takes $\hat Y/\bar H$ onto the intersection of two finite unions of balls of radius $\epsilon$ in $S^2$, hence the homology group $H_n(S(\hat Y/\bar H))$ is finitely generated. This proves that $\bar C$ verifies the property $(*)$.

Let now $\varphi:\bar X\to \R$ be a filtering function on $\bar X$. The quotient space $\bar X/\bar H$ is naturally endowed with a filtering function $\varphi^*:\bar X/\bar H\to\R$ defined by setting $\varphi^* (\omega)=\max_{x\in\omega}\varphi(x)$ for every orbit $\omega$ of $\bar H$. A filtering function $\bar\varphi^*:S(\bar X/\bar H)\to\R$ can be defined by setting $\bar\varphi^*(c)$ equal to the smallest $u\in\R$ such that the corresponding sublevel set ${\p^*}^{-1}\left((-\infty,u]\right)$ contains the image of each singular simplex involved in the reduced representation of $c\in S(\bar X/\bar H)$, provided that $c$ is non-null. We observe that $\bar \varphi=\bar\varphi^*\circ \bar F$, where $\bar\varphi:\bar C\to\R$ is the filtering function induced by $\varphi$ (Section~\ref{stability}). 

Let us consider the filtering function $\p':S^{2}\to\R$ that we can obtain by setting $\p'(x_1,x_2,x_3):=\max\{\p(x_1,x_2,x_3,1),\p(x_1,x_2,x_3,-1)\}$  for every $(x_1,x_2,x_3)\in S^2$.
Since $\varphi^*=\varphi'\circ f$, by using the chain isomorphism $\bar F$ it is easy to check that 
the persistent Betti number function $\rho_n^{\bar\p}$ equals 
the classical persistent Betti number function $\rho_n^{\p'}$.


We will show that $\bar G$-invariant persistent homology can be used to discriminate between two functions $\varphi,\psi:\bar X\to\R$ that cannot be distinguished by directly applying classical persistent homology. 
In order to illustrate this fact, let us consider the functions $\varphi,\psi:\bar X\to\R$ defined by setting 
$\varphi(x_1,x_2,x_3,x_4)=x_3$ and 
$$
\psi(x_1,x_2,x_3,x_4)= \left\{ \begin{array}{rl}
x_3 &\mbox{ if $x_4=1$} \\
-x_3 &\mbox{ if $x_4=-1$}
       \end{array} \right.
$$
for every $(x_1,x_2,x_3,x_4)\in \bar X$.

It is easy to check that a homeomorphism $g:\bar X\to \bar X$ exists, such that $\psi=\p\circ g$. In other words, $d_{\Homeo(\bar X)}(\p,\psi)=0$,
$\p$ and $\psi$ have the same persistent Betti number functions, and the direct application of classical persistent homology does not give a positive lower bound for $d_{\bar G}(\p,\psi)$. However, the persistent Betti number functions of the $\bar G$-chain complex $\bar C$ with respect to the induced filtering functions $\bar \p$ and $\bar \psi$ \emph{do not} coincide. This can be seen by computing the $\bar G$-invariant persistent homology in degree $1$. Indeed, 
 while the group $H_1\left(\bar C^{\bar \varphi\le t}\right)$ is trivial for every $t\in \R$, $H_1\left(\bar C^{\bar \psi\le t}\right)$ is isomorphic to $\K$ for $0\le t<1$.
This follows from the computation of the classical persistent Betti number functions $\rho_n^{\p'}$ and $\rho_n^{\psi'}$, where 
$\p'(x_1,x_2,x_3):=\max\{\p(x_1,x_2,x_3,1),\p(x_1,x_2,x_3,-1)\}=x_3$ and $\psi'(x_1,x_2,x_3):=\max\{\psi(x_1,x_2,x_3,1),\psi(x_1,x_2,x_3,-1)\}=|x_3|$, respectively,  for every $(x_1,x_2,x_3)\in S^2$.

On one hand, the persistence diagram associated with the persistent homology group in degree $1$ of the 
$\bar G$-chain complex $\bar C$ with respect to the filtering function $\bar \p$ is trivial. On the other hand, the persistence diagram associated with the persistent homology group in degree $1$ of the 
$\bar G$-chain complex $\bar C$ with respect to the filtering function $\bar \psi$ contains just the point $(0,1)$, with multiplicity $1$ (apart from the trivial points on the line $\{(u,v)\in\R^2:u=v\}$). 

It follows that the matching distance between the persistent Betti number functions in degree $1$ of the $\bar G$-chain complex $\bar C$ with respect to the filtering functions $\bar \p$ and $\bar \psi$ is at least the max-distance between the point $(0,1)$ and the line $\{(u,v)\in\R^2:u=v\}$, i.e. $1/2$. By applying Theorem~\ref{tG}, it follows that $d_{\bar G}(\p, \psi)\ge 1/2$, while both $d_{\Homeo(\bar X)}(\p, \psi)$
and the matching distance between the classical persistent Betti number functions of $\varphi$ and $\psi$ vanish. 

\end{ex}

\section{Discussion and further research}

Our method is grounded in the use of singular homology, in order to simplify the theoretical treatment. This fact constitutes a problem from the computational point of view. Actually, the use of simplicial homology would make our approach much more suitable for applications. 

The attempt of using simplicial homology in our framework leads to the need for a  triangulation of the topological space $X$ that is sufficiently fine, and invariant under the action of the group $H$ described in Subsection~\ref{generalization}, provided that $H$ acts freely on $X$. If such a triangulation is available, we can replace the previously considered singular chains that are left fixed under the action of $H$ with simplicial chains that are left fixed under the action of $H$. In other words, we can compute the simplicial homology of the quotient space $X/H$ via the quotient triangulation induced by the triangulation of $X$. In several application this is not difficult to do, since the space $X$ is fixed, and the search for an $H$-invariant triangulation
can be worth the effort. We also underline that, according to Subsection~\ref{generalization}, the group $H$ is finite. This fact makes the construction of an $H$-invariant triangulation much more affordable.

For example, if we are interested in comparing real-valued functions defined on $S^1$ with respect to the group of rotations (cf. Example~\ref{ex1}), it is quite easy to find a triangulation of $S^1$ that is invariant under the action of the group generated by the central symmetry. 
If we are interested in comparing real-valued functions defined on the topological space $\bar X$ with respect to the group $\bar G$ (see Example~\ref{ex3}), it is quite easy to find a triangulation of $\bar X$ that is invariant under the action of the group $\bar H$. 
We highlight that these triangulations \emph{do not} depend on the filtering functions and have to be computed just once.

Actually, in several applications where each filtering function is the result of a measurement, just one topological space $X$ is involved, and hence only one invariant triangulation is required to apply our method. 
For instance, we can refer to the problem of comparing the shapes of objects represented by clouds of points belonging to a fixed compact subset $B$ of a Euclidean space $\E^n$. In this case we can set $X=B$, while each filtering function $\varphi:B\to\R^k$ can describe both the  distance from the given cloud and other properties (cf., e.g., \cite{FrLa13}). Also in this case, just an invariant triangulation of $B$ is required.

We conclude this section by sketching a possible approach to the case that no $H$-invariant triangulation of $X$ is available, under the assumption that $H$ acts freely on $X$. For the sake of simplicity, we also assume that our filtering functions are real-valued. 

It is not restrictive to assume that $X$ is the body $|\Gamma|$ of a complex $\Gamma$ realized in $\R^n$.  
By possibly applying some barycentric subdivisions to $\Gamma$, we can also assume that its simplexes 
have diameters less than a given $\delta>0$. In general, the complex $\Gamma$ will not be invariant under the action of the group $H$.

Let us fix an $\epsilon>0$. For every $n\in \mathbb{N}$, let us consider the set of all ``elementary'' singular $n$-chains $c$  for which the following property holds: 
an ordered $r$-tuple $(\sigma_1,\ldots,\sigma_r)$ of singular $n$-simplexes belonging to $S(X)$ exists, such that $c=\sum_{i=1}^{r} \sigma_i$ and $\max_{p\in\Delta_n}\|h_i\circ\sigma_1(p)-\sigma_i(p)\|_\infty\le\epsilon$
for $1\le i\le r$.
We say that  these chains are \emph{almost symmetric} with respect to the group $H=\{h_1,\ldots,h_r\}$. We define $\bar C_n^\epsilon(X)$ to be the set of all linear combinations of these elementary almost symmetric singular $n$-chains. In this way we obtain a chain complex $\bar C^\epsilon(X)$. 
Let us consider the filtering function $\tilde \varphi:\bar C^\epsilon(X)\to\R$ that takes each non-null $n$-chain $c$ to the minimum value $t$ such that $c$ admits a representation $c=\sum_{j=1}^{m} a^j\sigma_{j}$ in $S_n(X)$, with $\sigma_j(\Delta_n)\subseteq X^{\varphi\le  t}$ for $1\le j\le m$. 
 
Moreover, for every $n\in \mathbb{N}$, let us consider the set of all ``elementary'' simplicial $n$-chains $c'$  for which the following property holds: 
an ordered $r$-tuple $(\sigma_1,\ldots,\sigma_r)$ of \emph{linear} singular $n$-simplexes exists, such that $c'=\sum_{i=1}^{r} \tau_i$ with $\tau_i\in \Gamma$, $|\tau_i|=\sigma_i(\Delta_n)$ and $\max_{p\in\Delta_n}\|h_i\circ\sigma_1(p)-\sigma_i(p)\|_\infty\le\epsilon$
for $1\le i\le r$.
Once again, we say that  these chains are \emph{almost symmetric} with respect to the group $H=\{h_1,\ldots,h_r\}$. We define $C_n^\epsilon$ to be the set of all linear combinations of these elementary almost symmetric simplicial $n$-chains. In this way we obtain a chain complex $C^\epsilon$. 
Let us consider on $C^\epsilon$ the filtering function $\hat \varphi$ that takes each non-null chain $c'$ to the minimum value $t$ such that $c'$ admits a representation $c'=\sum_{j=1}^{m} a^j\tau_{j}$ in the simplicial chain complex of $\Gamma$, with $|\tau_j|\subseteq X^{\varphi\le  t}$ for $1\le j\le m$. 

We notice that $\bar C^\epsilon$ and $C^\epsilon$ are chain complexes but not $G$-chain complexes, in general. 

The proof that a continuous non-negative function $\eta:\R^2\to \R$ exists such that $\eta(0,0)=0$ and the persistence modules of $\tilde\varphi$ and $\hat \varphi$ are $\eta(\delta,\epsilon)$-interleaved (cf. \cite{ChCo*09}) would imply that the persistent Betti number functions of the chain complex 
$C^\epsilon$ with respect to the filtering function $\hat \varphi$ are close to the persistent Betti number functions of the chain complex $\bar C^\epsilon$ with respect to the filtering function $\tilde \varphi$. This proof could be based on the simplicial approximation theorem.

Under suitable assumptions, it should then be possible to retrieve the persistent Betti number functions of the $G$-chain complex $\bar C$ (described at the beginning of  Subsection~\ref{generalization}) with respect to the filtering function $\bar \varphi$,
as the limit of the persistent Betti number functions of the chain complex $\bar C^\epsilon$ with respect to the filtering function $\tilde \varphi$, for $\delta$ and $\epsilon$ going to $0$. 
Ultimately, the persistent Betti number functions of the simplicial chain complex 
$C^\epsilon$  with respect to the filtering function $\hat \varphi$ should be a good approximation of the persistent Betti number functions of the $G$-chain complex $\bar C$ with respect to the filtering function $\bar \varphi$.

However, we think that this line of research is not trivial and deserves a separate and detailed treatment.  

Another interesting topic could be the one concerning the choice of the operator that takes each filtering function $\varphi$ on the topological space $X$ to the filtering function $\bar \varphi$ on the $G$-chain complex $\bar C$ (Section~\ref{stability}).
In the proof of Theorem~\ref{tG} we use the fact that this operator verifies the inequality $\sup_{c\in \bar C}\|\bar \varphi(c)-\bar \psi(c)\|_\infty\le\max_{x\in X}\|\varphi(x)-\psi(x)\|_\infty$. Other operators taking filtering functions on $X$ to filtering functions on $\bar C$ could be used in our method, provided that they verify the same inequality. 
For example, we could use the operator that takes each filtering function $\varphi$ on $X$ to the filtering function $\varphi^\sharp$ on $\bar C$ defined by setting $\varphi^\sharp(c):=\frac{1}{r}\sum_{1\le i\le r}\bar\varphi(h_i\circ \sigma)$, for every elementary singular chain $c=\sum_{i=1}^r h_i\circ \sigma$, provided that $G$ acts freely on $X$.

Finally, it would be also interesting to determine if other techniques to construct filtered $G$-chain complexes exist which are essentially different from the one described 
in Subsection~\ref{generalization}.

In conclusion, the general method that we have sketched in this article probably requires a great amount of further research, from the algebraic, homological and computational point of view. 

We postpone the treatment of these issues to subsequent papers.

\section*{Acknowledgment}
The author thanks Silvia Biasotti, Fr\'ed\'eric Chazal, Herbert Edelsbrunner, Massimo Ferri, Grzegorz Jab\l o\'nski, Claudia Landi, Michael Lesnick, Marian Mrozek, Michele Mulazzani and an anonymous referee for their suggestions and advice, and the Leibniz Center for Informatics in Dagstuhl for its inspiring hospitality. 

A special thank to Peter Landweber, for his invaluable help. 

The research described in this  article has been partially supported by GNSAGA-INdAM (Italy), and is based on the work realized by the author within the ESF-PESC Networking Programme ``Applied and Computational Algebraic Topology''.

This paper is dedicated to the beloved memory of Don~Renato~Gargini.

\bibliographystyle{model1-num-names}

\end{document}